\documentclass[12pt,reqno]{amsart}

\usepackage{amssymb,hyperref}

\newtheorem{theorem}{Theorem}[section]
\newtheorem{proposition}[theorem]{Proposition}

\newtheorem{corollary}[theorem]{Corollary}

\theoremstyle{definition}

\numberwithin{equation}{section}

\begin{document}

\baselineskip=15pt

\title[Deformation theory of Cartan geometries, II]{Deformation theory of holomorphic Cartan geometries, II}

\author[I. Biswas]{Indranil Biswas}

\address{School of Mathematics, Tata Institute of Fundamental
Research, Homi Bhabha Road, Mumbai 400005, India}

\email{indranil@math.tifr.res.in}

\author[S. Dumitrescu]{Sorin Dumitrescu}

\address{Universit\'e C\^ote d'Azur, CNRS, LJAD, France}

\email{dumitres@unice.fr}

\author[G. Schumacher]{Georg Schumacher}

\address{Fachbereich Mathematik und Informatik, Philipps-Universit\"at
Marburg, Lahnberge, Hans-Meerwein-Strasse, D-35032 Marburg, Germany}

\email{schumac@mathematik.uni-marburg.de}



\date{}

\begin{abstract}
In this continuation of \cite{BDS}, we investigate the deformations of holomorphic Cartan geometries where the underlying
complex manifold is allowed to move. The space of infinitesimal deformations of a flat holomorphic Cartan geometry is computed. We show
that the natural forgetful map, from the infinitesimal deformations of a flat holomorphic Cartan geometry to the
infinitesimal deformations of the underlying flat principal bundle on the
topological manifold, is an isomorphism.
\end{abstract}

\maketitle

\section{Introduction}\label{se1}

In \cite{BDS} we studied the deformations of holomorphic Cartan geometries on a fixed 
compact complex manifold. Here we consider the more general deformations of holomorphic 
Cartan geometries where the underlying compact complex manifold is allowed to move.
We also investigate the deformations of the flat holomorphic Cartan geometries.

Let $G$ be a connected Lie group and $H\, \subset\, G$ a 
connected closed Lie subgroup. As a consequence of the foundational work of Cartan and Ehresmann, a flat Cartan 
geometry with model $(G, \, H)$ on a compact manifold $M$ is determined by the following 
geometrical objects: a smooth principal $G$--bundle $E_G$ over $M$ endowed with a flat 
connection and a principal $H$--subbundle $E_H\, \subset\, E_G$ transverse to the integrable horizontal
distribution associated to the flat connection \cite{Eh} (see also the survey \cite{BD1}).

Fix a base point $x_0\, \in\, M$ and a point $z\, \in\, (E_H)_{x_0}$ in the fiber of $E_H$ over $x_0$.
The above properties imply that the pull-back of this principal $G$--bundle $E_G$ to the universal 
cover $\widetilde{M}$ of $M$ for $x_0$ is isomorphic to the trivial principal bundle $\widetilde{M} \times 
G\,\longrightarrow\, \widetilde{M}$ and the pull-back of $E_H$ to $\widetilde{M}$ is defined by a smooth map 
$\widetilde{M} \,\longrightarrow\, G/H$. The above mentioned transversality condition is equivalent to 
the statement that this map $\widetilde{M} \,\longrightarrow\, G/H$ is a local diffeomorphism; it
is customary to call this map $\widetilde{M} \,\longrightarrow\, G/H$ the 
developing map of the (flat) Cartan geometry. The flat Cartan geometry on $M$ with model $(G,\, H)$ produces
a monodromy homomorphism $\rho \,:\, \pi_1(M,\, x_0) \,\longrightarrow\, G$.
The developing map is $\pi_1(M,\, x_0)$--equivariant with respect to the action of $\pi_1(M,\, x_0)$ on $\widetilde{M}$
via the deck transformations and the action of $\pi_1(M,\, x_0)$ on $G/H$ 
via the monodromy morphism $\rho$ and the left-translation action of $G$ on $G/H$ \cite{Eh} 
(see the expository works \cite{Sh,BD1}).

The above geometrical description of Ehresmann leads to the so-called 
Ehresmann-Thurston principle which states that the Riemann-Hilbert map associating to 
each flat Cartan geometry its monodromy morphism $\rho \,:\, \pi_1(M)\, 
\longrightarrow\, G$ (uniquely determined up to inner automorphisms of $G$) is a local 
homeomorphism between the moduli space of flat Cartan geometries with model 
$(G,\,H)$ on $X$ and the space of group homomorphisms from $\pi_1(M)$ to $G$ (modulo the 
action of $G$ acting on the target $G$ by inner conjugation) \cite[p.~115, Theorem 2.1]{BG} (see also \cite{CEG}).

Now let $G$ be a connected complex Lie group and $H\, \subset\, G$ a connected complex closed Lie subgroup. Then the model manifold $G/H$ 
inherits a $G$-invariant complex structure. Any flat Cartan geometry with model $(G,\,H)$ induces on $M$ an underlying complex structure (for 
which the above $G$-bundle $E_G$, its flat connection and the transverse $H$-subbundle $E_H$ are all holomorphic). Hence there is a natural 
forgetful map from the deformation space of flat Cartan geometries with model $(G,\,H)$ to the Kuranishi space of $M$. In the particular case 
where $M$ is a surface and $G\,=\,{\rm PSL}(2, \mathbb C)$, with $H \subset G$ being the stabilizer of a point in the complex projective 
line, a flat Cartan geometry with model $(G,\,H)$ determines a complex projective structure on $S$ and hence an underlying structure of 
Riemann surface. Led by the work of Klein and Poincar\'e, the complex projective structures had a major role in the formulation and (some of) 
the proofs of the uniformization theorem for Riemann surfaces (see \cite{Gu} or \cite{StG}). Indeed, the uniformization theorem for Riemann 
surfaces ensures the existence on any Riemann surface of a compatible complex projective structure (meaning it defines the same complex 
structure) with injective developing map.

Much more recently, the deformation space of flat Cartan geometries with model \linebreak $G\,=\,{\rm SL}(2, \mathbb C) \times {\rm SL}(2, 
\mathbb C)$ and $H\,=\,{\rm SL}(2, \mathbb C)$, diagonally embedded in $G$, and the associated forgetful map to the Kuranishi space were studied by 
Ghys in \cite{Gh}. Using this method, Ghys computed in \cite{Gh} the Kuranishi space of the parallelizable manifolds ${\rm SL}(2, \mathbb C) 
/ \Gamma$, where $\Gamma$ is a uniform lattice in ${\rm SL}(2, \mathbb C)$. Ghys proved that the forgetful map realizes an isomorphism 
between the deformation space of ${\rm SL}(2, \mathbb C) / \Gamma$ as flat Cartan geometry with the above model $(G,\,H)$ and the Kuranishi 
space of the complex manifold ${\rm SL}(2, \mathbb C) / \Gamma$. Moreover, the deformation space of ${\rm SL}(2, \mathbb C) / \Gamma$ as flat 
Cartan geometry is modeled on the germ, at the trivial morphism, of the algebraic variety of group homomorphisms from $\Gamma$ into ${\rm 
SL}(2, \mathbb C)$. In particular, for any uniform lattice $\Gamma$ with positive first Betti number this germ has positive dimension. Hence 
the corresponding parallelizable manifolds ${\rm SL}(2, \mathbb C) / \Gamma$ admit nontrivial deformations of the underlying complex 
structure. Those examples of flexible parallelizable manifolds associated to semi-simple complex Lie groups are exotic (by Raghunathan's 
rigidity results \cite{Ra} a compact quotient of a complex Lie group by a lattice has a rigid complex structure, if no local factor is 
isomorphic to ${\rm SL}(2, \mathbb C)$).

Our result below is a reformulation of the Ehresmann-Thurston principle in the context of flat holomorphic Cartan geometries with model $(G, \, H)$, where $G$ is a connected complex Lie group and $H\, \subset\, G$ a connected
complex closed Lie subgroup. 

Let $E_H$ be a holomorphic principal $H$-bundle on
a compact complex manifold $M$; the holomorphic principal $G$--bundle on $M$
obtained by extending the structure group of $E_H$ using the inclusion map
$H\, \hookrightarrow\, G$ will be denoted by $E_G$. Let
$$\vartheta\, :\,\text{At}(E_H)\, \stackrel{\sim}\longrightarrow\,\text{ad}(E_G)$$
be a flat holomorphic Cartan geometry on $M$ modeled on the pair $(G,\, H)$, where $\text{At}(E_H)$
is the Atiyah bundle for $E_H$ and $\text{ad}(E_G)$ is the adjoint bundle for $E_G$. The
isomorphism $\vartheta$ produces a flat holomorphic connection $\theta'$ on the principal $G$--bundle $E_G$.
Let
$$
{\mathcal I}_{CG}
$$
denote the space of all
infinitesimal deformations of the flat holomorphic Cartan geometry $(M,\, E_H,\, \theta)$
in the category of the flat holomorphic Cartan geometries.
Let
$$
{\mathcal I}_{FC}
$$
denote the space of all
infinitesimal deformations of the flat principal $G$--bundle $(E_G,\, \theta')$ on
the topological manifold $M$, where $\theta'$ is the above mentioned flat connection
on $E_G$ given by $\vartheta$. Associating to any flat holomorphic Cartan geometry $(E_H,\, \theta)$
the corresponding flat principal $G$--bundle on the underlying topological manifold we obtain
 a homomorphism
$$
\varphi\, :\, {\mathcal I}_{CG}\, \longrightarrow\, {\mathcal I}_{FC}\, .
$$
(see \eqref{icg} and \eqref{ifc}).

We prove the following (see Theorem \ref{thm1}):

\begin{theorem}
The above homomorphism $\varphi$ is an isomorphism.
\end{theorem}

More generally, we study here the deformation space of (non necessarily flat) holomorphic Cartan geometries where the underlying complex 
structure of the manifold is allowed to move. We generalize in this broader context results previously obtained in \cite{BDS} (see 
Proposition \ref{prop2} and Corollary \ref{cor1}).

It should be mentioned that Proposition \ref{prop2} and Corollary \ref{cor1} correct and 
generalize Theorem 3.4 in \cite{BDS}\footnote{We thank Yasuhiro Wakabayashi who pointed 
out that the kernel of the exact sequence in the statement of the Theorem 3.4 is not the 
correct one (see \cite[Remark 6.4.5]{Wa}). It should be replaced with $H^0(X,\, Hom(TX 
,\, {\rm ad}(E_G)))$.}.

\section{Cartan geometry}

The holomorphic tangent (respectively, cotangent) bundle of a complex manifold $N$ will
be denoted by $TN$ (respectively, $\Omega^1_N$).

Let $G$ be a connected complex Lie group; its Lie algebra will be denoted by
$\mathfrak g$. Let $H\, \subset\, G$ be a connected
complex closed Lie subgroup with Lie algebra $\mathfrak h\, \subset\,\mathfrak g$.

Take a connected complex manifold $M$. Let
\begin{equation}\label{e1}
f\,: E_H\, \longrightarrow\, M
\end{equation}
be a holomorphic principal $H$--bundle on $M$. The action of $H$ on $E_H$ produces
a holomorphic action of $H$ on the holomorphic tangent bundle $TE_H$. Let
$$
df\, :\, TE_H\, \longrightarrow\, f^*TM
$$
be the differential of the projection $f$ in \eqref{e1}. Using the action of
$H$ on $E_H$, the kernel of $df$ is identified with the trivial holomorphic vector
bundle on $E_H$ with fiber $\mathfrak h$; such an identification is known as the
Maurer-Cartan form.

A holomorphic Cartan geometry on $M$ of type $G/H$ is a pair $(E_H,\, \theta)$, where
$E_H$ is a holomorphic principal $H$--bundle on $M$ and
\begin{equation}\label{e0}
\theta\, :\, TE_H\, \longrightarrow\, E_H\times {\mathfrak g}
\end{equation}
is a holomorphic homomorphism, such that
\begin{enumerate}
\item $\theta$ is an isomorphism,

\item $\theta$ is $H$--equivariant for the above action of $H$ on $TE_H$ and the
diagonal action of $H$ on $E_H\times {\mathfrak g}$ constructed using the
adjoint action of $H$ on ${\mathfrak g}$ and the above mentioned action of $H$ on $E_H$, and

\item the restriction of $\theta$ to $\text{kernel}(df)$ coincides with the
above identification of $\text{kernel}(df)$ with $E_H\times \mathfrak h$.
\end{enumerate}
(See \cite{Sh}.)

Since $\theta$ in \eqref{e0} is $H$--equivariant, it descends to a homomorphism between 
appropriate vector bundles over $M$. We now recall an equivalent reformulation of the above 
definition.

As in \eqref{e1}, let $E_H$ be a holomorphic principal $H$--bundle on $M$.
The action of $H$ on $TE_H$ produces an action of $H$ on the direct image $f_*TE_H$
over the trivial action of $H$ on $M$. Let
\begin{equation}\label{j1}
{\rm At}(E_H) \,=\, (f_*TE_H)^H \,=\, (TE_H)/H\, \longrightarrow\, M
\end{equation}
be the Atiyah bundles for $E_H$ \cite[p.~187, Theorem 1]{At}, where
$(f_*TE_H)^H\, \subset\, f_*TE_H$ is the $H$--invariant subbundle. Let
\begin{equation}\label{ead}
\text{ad}(E_H)\, :=\, (f_*\text{kernel}(df))^H\,=\, \text{kernel}(df)/H
\, \subset\, (TE_H)/H\,=\, {\rm At}(E_H)
\end{equation}
be the adjoint bundle of $E_H$. So we have a short exact sequence
of holomorphic vector bundles on $M$
\begin{equation}\label{e-1}
0\, \longrightarrow\, {\rm ad}(E_H)\, \, \stackrel{h_1}{\longrightarrow}\, {\rm At}(E_H)
\, \stackrel{df}{\longrightarrow}\, TM\, \longrightarrow\,0
\end{equation}
which is known as the \textit{Atiyah exact sequence} for $E_H$. We recall that
a holomorphic connection on $E_H$ is a holomorphic splitting of the
exact sequence in \eqref{e-1}, meaning a holomorphic homomorphism $\psi\, :\,
TM\, \longrightarrow\, \text{At}(E_H)$ such that $(df)\circ\psi\,=\, {\rm Id}_{TM}$, where
$df$ is the projection in \eqref{e-1} (see \cite[p~188, Definition]{At}).

Let $E_H(\mathfrak h)$ (respectively, $E_H(\mathfrak g)$) be the holomorphic
vector bundle over $M$ associated to the principal $H$--bundle $E_H$ for the
adjoint action of $H$ on $\mathfrak h$ (respectively, $\mathfrak g$). We note that
$E_H(\mathfrak h)$ is the adjoint vector bundle $\text{ad}(E_H)$. Let
\begin{equation}\label{e2}
E_G\, :=\, E_H\times^H G \, \longrightarrow\, M
\end{equation}
be the holomorphic principal $G$--bundle on $M$ obtained by extending the
structure group of $E_H$ using the inclusion map of $H$ in $G$. The above vector bundle
$E_H(\mathfrak g)$ evidently coincides with the adjoint bundle $\text{ad}(E_G)$ for $E_G$.
The inclusion map ${\mathfrak h}\, \hookrightarrow\, \mathfrak g$ produces a
short exact sequence of holomorphic vector bundles on $M$
\begin{equation}\label{e3}
0\, \longrightarrow\, {\rm ad}(E_H)\, \, \stackrel{h_2}{\longrightarrow}\, {\rm ad}(E_G)
\, \longrightarrow\, {\rm ad}(E_G)/{\rm ad}(E_H)\, \longrightarrow\,0\, .
\end{equation}

A holomorphic Cartan geometry on $M$ of type $G/H$ is a pair $(E_H,\, \vartheta)$, where
$E_H$ is a holomorphic principal $H$--bundle on $M$ and
\begin{equation}\label{e-3}
\vartheta\, :\, {\rm At}(E_H)\, \longrightarrow\, {\rm ad}(E_G)
\end{equation}
is a holomorphic isomorphism of vector bundles, such that
\begin{equation}\label{e-2}
\vartheta\circ h_1 \,=\,h_2\, ,
\end{equation}
where $h_1$ and $h_2$ are the homomorphisms in \eqref{e-1} and \eqref{e3} respectively.

Using \eqref{j1} it is straightforward to check that the above definition of a
holomorphic Cartan geometry on $M$ of type $G/H$ is equivalent to the definition
given earlier.

For any isomorphism $\vartheta$ as in \eqref{e-3} satisfying the equation in \eqref{e-2}, we have 
the following commutative diagram
\begin{equation}\label{e12}
\begin{matrix}
0 & \longrightarrow & {\rm ad}(E_H) & \stackrel{h_1}{\longrightarrow} & {\rm At}(E_H)
& \longrightarrow & TM & \longrightarrow & 0\\
&& \Vert && ~\Big\downarrow\vartheta && \Big\downarrow\\
0 & \longrightarrow & {\rm ad}(E_H) & \stackrel{\iota_1}{\longrightarrow} & {\rm ad}(E_G)
& \longrightarrow & {\rm ad}(E_G)/{\rm ad}(E_H) & \longrightarrow & 0
\end{matrix}
\end{equation}
\cite[Ch.~5]{Sh}; the above homomorphism $TM\,\longrightarrow\,{\rm ad}(E_G)/{\rm ad}(E_H)$
induced by $\vartheta$ is an isomorphism because $\vartheta$ is so.

Let
\begin{equation}\label{e-4}
0\, \longrightarrow\, {\rm ad}(E_G)\, \, \stackrel{h_3}{\longrightarrow}\, {\rm At}(E_G)\, :=\, (TE_G)/G
\, \longrightarrow\, TM\, \longrightarrow\,0
\end{equation}
be the Atiyah exact sequence for the principal $G$--bundle $E_G$ in \eqref{e2}.
The Atiyah bundle ${\rm At}(E_G)$ in \eqref{e-4}
can also be constructed using ${\rm At}(E_H)$ and ${\rm ad}(E_G)$; we now recall this
construction. Consider the embedding
\begin{equation}\label{em}
\text{ad}(E_H)\,\hookrightarrow\, {\rm At}(E_H)\oplus {\rm ad}(E_G)
\end{equation}
that sends any $v$ to $(h_1(v),\, -h_2(v))$, where $h_1$ and $h_2$ are the
homomorphisms in \eqref{e-1} and \eqref{e3} respectively. The Atiyah bundle 
${\rm At}(E_G)$ is the quotient bundle
\begin{equation}\label{b2}
{\rm At}(E_G)\,=\, ({\rm At}(E_H)\oplus {\rm ad}(E_G))/\text{ad}(E_H)
\end{equation} 
for this embedding. The map $h_3$ in 
\eqref{e-4} is given by the inclusion $h_2$ or $h_1$ of $\text{ad}(E_G)$ in
${\rm At}(E_H)\oplus {\rm ad}(E_G)$; note that they produce the same homomorphism to
the above quotient bundle $({\rm At}(E_H)\oplus {\rm ad}(E_G))/\text{ad}(E_H)$.

Let $(E_H,\, \vartheta)$ be a holomorphic Cartan geometry of type $G/H$ on $M$.
Then the homomorphism
$$
{\rm At}(E_H)\oplus {\rm ad}(E_G)\, \longrightarrow\, {\rm ad}(E_G),\, \ \
(v,\, w) \, \longmapsto\, \vartheta(v)+w
$$
produces a homomorphism
\begin{equation}\label{e10}
\theta'\, :\, {\rm At}(E_G)\,=\, ({\rm At}(E_H)\oplus
{\rm ad}(E_G))/\text{ad}(E_H)\, \longrightarrow\, {\rm ad}(E_G)\, ,
\end{equation}
because it vanishes on the subbundle $\text{ad}(E_H)$ in \eqref{em}; see \eqref{b2}.
It is straightforward to check that $\theta'\circ h_3\,=\, \text{Id}_{{\rm 
ad}(E_G)}$, where $h_3$ is the homomorphism in \eqref{e-4}. Consequently,
$\theta'$ gives a holomorphic splitting of the exact sequence in \eqref{e-4}. Therefore, 
$\theta'$ is a holomorphic connection on the principal $G$--bundle $E_G$ \cite{At}.

The curvature $\text{Curv}(\theta')$ of the connection $\theta'$ is a holomorphic section
$$
\text{Curv}(\theta')\, \in\, H^0(X,\, \text{ad}(E_G)\otimes \Omega^2_M)\, ,
$$
where $\Omega^i_M\, :=\, \bigwedge\nolimits^i (TM)^*$.

The Cartan geometry $(E_H,\, \vartheta)$ is called \textit{flat} if
$$
\text{Curv}(\theta')\,=\,0
$$
\cite[Ch.~5, \S~1, p.~177]{Sh}.

\section{Connection and differential operators}

Take a connected complex manifold $M$. Let
$$
f\,: E_H\, \longrightarrow\, M
$$
be a holomorphic principal $H$--bundle, and let
$$
\theta\, :\, TE_H\, \longrightarrow\, E_H\times {\mathfrak g}
$$
be a holomorphic isomorphism defining a holomorphic
Cartan geometry of type $G/H$ on $M$; see \eqref{e0}. Take
a nonempty open subset $U\, \subset\, M$. Let
$$
\varpi\, \in \, H^0\left(f^{-1}(U),\, (TE_H)\big\vert_{f^{-1}(U)}\right)^H \,=\,
H^0\big(f^{-1}(U),\, T(f^{-1}(U))\big)^H
$$
be an $H$--invariant holomorphic vector field on $f^{-1}(U)$. So we have
\begin{equation}\label{e4}
\theta(\varpi)\, \in \, H^0(f^{-1}(U),\, f^{-1}(U)\times {\mathfrak g})\,=\,
H^0\left(f^{-1}(U),\, (E_H\times {\mathfrak g})\big\vert_{f^{-1}(U)}\right)\, .
\end{equation}
In other words, $\theta(\varpi)$ is a ${\mathfrak g}$--valued holomorphic function
on $f^{-1}(U)$. Take any point $z\, \in\, f^{-1}(U)$, and also take a holomorphic
tangent vector
$$
v\, \in\, T_zE_H\, .
$$
For the function $\theta(\varpi)$ in \eqref{e4}, consider
\begin{equation}\label{e6}
v(\theta(\varpi))+ [\theta(v),\, \theta(\varpi)]\, \in\, {\mathfrak g}
\end{equation}
for the above tangent vector $v$.

First treat the case where $v\, \in\, \text{kernel}(df)(z)$; here $f$, as before,
is the projection of $E_H$ to $M$. Recall that the third condition in the definition
of a Cartan geometry says that the restriction of $\theta$ to $\text{kernel}(df)$
coincides with the identification of $\text{kernel}(df)$ with $E_H\times \mathfrak h$
given by the action of $H$ (the Maurer--Cartan form).

Since $\theta$ is $H$--equivariant, and $\varpi$ is $H$--invariant, it follows that
$\theta(\varpi)$ transforms as the inverse adjoint representation, of $H$
on $\mathfrak g$, under the $H$--action.
So under the infinitesimal $\mathfrak h$--action it equals $-[\theta(v),\, \theta(\varpi)]$
for the action of $v\, \in\, \mathfrak h$. This implies that
\begin{equation}\label{e5}
v(\theta(\varpi))+ [\theta(v),\, \theta(\varpi)]\,=\, 0\, ,
\end{equation}
because $v\, \in\, \text{kernel}(df)(z)$.

Next we investigate the action of $H$ on the construction in \eqref{e6}. Take any
$g\, \in\, H$. Denote
$$
z'\,=\, zg\, ,
$$
and let
$$
v'\, =\, g\cdot v\, \in\, T_{z'}E_H
$$
be the image of $v$ under the automorphism of $TE_H$ given by the action of $g$.
It is straightforward to check that
\begin{equation}\label{e7}
v'(\theta(\varpi))+ [\theta(v'),\, \theta(\varpi)]
\,=\, \text{Ad}(g^{-1})(v(\theta(\varpi))+ [\theta(v),\, \theta(\varpi)])\, ;
\end{equation}
to clarify, $\text{Ad}(g)$ is the automorphism of $\mathfrak g$ corresponding to the
automorphism of $G$ defined by $y\, \longmapsto\, gyg^{-1}$.

{}From the combination of \eqref{e5} and \eqref{e7} it follows that
the construction in \eqref{e6} produces a homomorphism of sheaves
\begin{equation}\label{e8}
\Theta\, :\, \text{At}(E_H)\, \longrightarrow\, \text{ad}(E_G)\otimes\Omega^1_M\, .
\end{equation}
To explain this homomorphism $\Theta$, take $\varpi$ and $v$ as in \eqref{e6}.
Then $\varpi$ gives a section of $\text{At}(E_H)\big\vert_U$ (see \eqref{j1}),
which will be denoted by $\widetilde{\varpi}$; also, we have $df(v)\, \in\, T_{f(z)} M$, where, as
before, $df$ is the differential of the projection $f$. The homomorphism
$\Theta$ in \eqref{e8} is uniquely determined by the condition that the element
$$
\Theta(\widetilde{\varpi})(df(v))\, \in\, \text{ad}(E_G)_{f(z)}
$$
coincides with the image of $(z,\, v(\theta(\varpi))+ [\theta(v),\, \theta(\varpi)])$ in
$\text{ad}(E_G)_{f(z)}$, where $v(\theta(\varpi))+ [\theta(v),\, \theta(\varpi)]
\, \in\, \mathfrak g$ is the element in \eqref{e6}. Recall that
$\text{Ad}(E_G)$ is the quotient of $E_H\times\mathfrak g$ where two elements $(z_1,\, v_1)$
and $(z_2,\, v_2)$ of $E_H\times\mathfrak g$ are identified if there is an element
$g\, \in\, H$ such that $z_2\, =\, z_1g$ and $v_2\,=\, \text{Ad}(g^{-1})(v_1)$, and hence
$(z,\, v(\theta(\varpi))+ [\theta(v),\, \theta(\varpi)])$ gives an element of $\text{ad}(E_G)_{f(z)}$. From
\eqref{e5} and \eqref{e7} it follows that the above characterization of
$\Theta(\widetilde{\varpi})(df(v))\, \in\, \text{ad}(E_G)_{f(z)}$
uniquely defines the homomorphism in \eqref{e8}.

{}From the construction of $\Theta$ in \eqref{e8} it follows immediately that
$\Theta$ is a holomorphic differential operator of order one.

The Cartan geometry $(E_H,\, \theta)$ produces a holomorphic isomorphism
$$
\vartheta\, :\, {\rm At}(E_H)\, \longrightarrow\, {\rm ad}(E_G)
$$
(see \eqref{e-3}). Let
\begin{equation}\label{e9}
\Theta\circ\vartheta^{-1}\, :\, {\rm ad}(E_G)\, \longrightarrow\,\text{ad}(E_G)\otimes\Omega^1_M
\end{equation}
be the composition of $\vartheta^{-1}$ with the homomorphism $\Theta$ in \eqref{e8}.

Recall the holomorphic connection $\theta'$ on $E_G$ in \eqref{e10}. Any holomorphic
connection on $E_G$ induces a holomorphic connection on any holomorphic fiber bundle
associated to $E_G$. In particular, the connection $\theta'$ produces
a holomorphic connection 
\begin{equation}\label{e11}
\Theta'\, :\, {\rm ad}(E_G)\, \longrightarrow\,\text{ad}(E_G)\otimes\Omega^1_M
\end{equation}
on the adjoint bundle ${\rm ad}(E_G)$.

\begin{proposition}\label{prop1}
The holomorphic connection $\Theta'$ in \eqref{e11} coincides with the homomorphism
$\Theta\circ\vartheta^{-1}$ constructed in \eqref{e9}.
\end{proposition}

\begin{proof}
We first recall the construction of the homomorphism $\Theta'$ in
\eqref{e11}. As in \eqref{e2}, let
$$
q\, :\, E_G \, \longrightarrow\, M
$$
be the holomorphic principal $G$--bundle on $M$ obtained by extending the
structure group of $E_H$ using the inclusion of $H$ in $G$.
Using the action of $G$ on $E_G$, the vector bundle ${\rm kernel}(dq)\, \subset\, TE_G$,
where $dq\, :\, TE_G\, \longrightarrow\, q^* TM$ is the differential of the projection $q$,
is identified with the trivial vector bundle $E_G\times{\mathfrak g}\,
\longrightarrow\, E_G$. Take any
$$
s\, \in\, H^0\left(U,\, {\rm ad}(E_G)\big\vert_U\right),
$$
where $U\, \subset\, M$ is a nonempty open subset. We recall that
$$
\text{ad}(E_G)\,=\, (q_*\text{kernel}(dq))^G\,=\, \text{kernel}(dq)/G
$$
(see \eqref{ead}). Using this isomorphism $$\text{ad}(E_G)\,=\, \text{kernel}(dq)/G
\, \subset\, (TE_G)/G \,=\, \text{At}(E_G),$$ the above section
$s$ of ${\rm ad}(E_G)\big\vert_U$ produces a holomorphic vector field
\begin{equation}\label{ts}
\widetilde{s}\, \in\, H^0\left(q^{-1}(U),\, (TE_G)\big\vert_{q^{-1}(U)}\right)\, .
\end{equation}
We note that $\widetilde{s}$ satisfies the following two conditions:
\begin{itemize}
\item $\widetilde{s}$ lies in the subspace
$$
H^0\left(q^{-1}(U),\, {\rm kernel}(dq)\big\vert_{q^{-1}(U)}\right)\,
\subset\, H^0\left(q^{-1}(U),\, (TE_G)\big\vert_{q^{-1}(U)}\right)\, ,
$$
where $dq$ as before is the differential of the projection $q$, and

\item the action of $G$ on $E_G$ preserves $\widetilde{s}$.
\end{itemize}
Take a holomorphic vector field
$$
v\, \in\, H^0(U, \, TU)\, .
$$
Let
$$
\widetilde{v}\, \in\, H^0\left(q^{-1}(U),\, (TE_G)\big\vert_{q^{-1}(U)}\right)
$$
be the horizontal lift of $v$ for the holomorphic connection $\theta'$ on $E_G$
in \eqref{e10}. Now consider the Lie bracket
$$
[\widetilde{v},\, \widetilde{s}]\, \in\, H^0\left(q^{-1}(U),\,
(TE_G)\big\vert_{q^{-1}(U)}\right)\, .
$$
It is straightforward to verify the following statements:
\begin{itemize}
\item The vector field $[\widetilde{v},\, \widetilde{s}]$ is $G$--invariant.
Indeed, this is a consequence of the fact that both $\widetilde{v}$ and $\widetilde{s}$
are $G$--invariant vector fields.

\item We have $$[\widetilde{v},\, \widetilde{s}]\, \in\, H^0\left(q^{-1}(U),\,
{\rm kernel}(dq)\big\vert_{q^{-1}(U)}\right)\, .$$ Indeed, this follows from the facts
that $\widetilde{s}\, \in\, H^0\left(q^{-1}(U),\,
{\rm kernel}(dq)\big\vert_{q^{-1}(U)}\right)$ and $\widetilde{v}$ is $G$--invariant.

\item The equality $$[\widetilde{\beta\cdot v},\, \widetilde{s}]\,=\, (\beta\circ q)
\cdot [\widetilde{v},\, \widetilde{s}]$$ holds for any holomorphic function $\beta$ on $U$. To
see this, note that $\widetilde{s}(\beta\circ q)\,=\, 0$ because
$$\widetilde{s}\, \in\, H^0\left(q^{-1}(U),\,
{\rm kernel}(dq)\big\vert_{q^{-1}(U)}\right).$$
The above equality follows immediately from this.

\item $[\widetilde{v},\, \widetilde{\beta\cdot s}]\,=\, (\beta\circ q)
\cdot [\widetilde{v},\, \widetilde{s}]+ (v(\beta)\circ q)\cdot \widetilde{s}$
for any holomorphic function $\beta$ on $U$.
\end{itemize}
Note that the first two of the above four statements together imply that $[\widetilde{v},\, \widetilde{s}]$ gives
a holomorphic section of ${\rm ad}(E_G)\big\vert_U$; this
holomorphic section of ${\rm ad}(E_G)\big\vert_U$ will be denoted by
$\widehat{[\widetilde{v},\, \widetilde{s}]}$.

The homomorphism $\Theta'$ in \eqref{e11} is uniquely determined by the following equation:
\begin{equation}\label{dc}
\langle \Theta'(s),\, v\rangle\,=\, \widehat{[\widetilde{v},\, \widetilde{s}]}\, ,
\end{equation}
where $\langle -, \, -\rangle$ denotes the contraction of forms by vector fields; note that
both sides of \eqref{dc} are sections of ${\rm ad}(E_G)\big\vert_U$. The third and fourth
statements above imply that $\Theta'$ defined by \eqref{dc} is actually a connection on $\text{ad}(E_G)$.

Using the earlier mentioned identification between ${\rm kernel}(dq)$
and the trivial vector bundle $E_G\times{\mathfrak g}\,
\longrightarrow\, E_G$, the vertical vector field $\widetilde{s}$ in \eqref{ts}
(vertical for the projection $q$) defines a
$\mathfrak g$--valued holomorphic function on $q^{-1}(U)$. This
$\mathfrak g$--valued holomorphic function on $q^{-1}(U)$ will be denoted by
$\widetilde{s}_1$. Consider the Lie bracket of vector fields
$$
[\widetilde{v},\, \widetilde{s}]\, \in\, H^0\left(q^{-1}(U),\,
{\rm kernel}(dq)\big\vert_{q^{-1}(U)}\right);
$$
it defines a $\mathfrak g$--valued holomorphic function on $q^{-1}(U)$.
This $\mathfrak g$--valued holomorphic function on $q^{-1}(U)$
coincides with
the derivative $\widetilde{v}(\widetilde{s}_1)$
of the $\mathfrak g$--valued function $\widetilde{s}_1$ in the direction of the vector
field $\widetilde{v}$.

Since $E_G$ is the principal $G$--bundle obtained by extending the structure group of
the principal $H$--bundle $E_H$ using the inclusion map $H\, \hookrightarrow\, G$, we have
an inclusion map
\begin{equation}\label{b1}
\text{At}(E_H)\, \hookrightarrow\, \text{At}(E_G)\, ;
\end{equation}
this inclusion map sends any $v\, \in\, \text{At}(E_H)$ to
$(v,\, 0)\, \in\, ({\rm At}(E_H)\oplus {\rm ad}(E_G))/\text{ad}(E_H)$; see \eqref{b2}.
The connection $\theta'$ on $E_G$ in \eqref{e10} produces a holomorphic splitting
$$
\theta'\, :\, \text{At}(E_G)\, \longrightarrow\, \text{ad}(E_G)
$$
(see \eqref{e10}). Combining this with the homomorphism in \eqref{b1}, we get a homomorphism
\begin{equation}\label{b3}
\text{At}(E_H)\, \longrightarrow\, \text{ad}(E_G)\, .
\end{equation}
{}From the construction of the connection $\theta'$ in \eqref{e10} it follows immediately that
this homomorphism in \eqref{b3} coincides with the homomorphism $\vartheta$ in \eqref{e-3}.

Let $w\, \in\, H^0(q^{-1}(U),\, TE_G\big\vert_{q^{-1}(U)})^G$ 
be a holomorphic vector field on $q^{-1}(U)\, \subset\, E_G$. As before, take a holomorphic vector
field $v\, \in\, H^0(U, \, TU)$, and let
$$
\widetilde{v}\, \in\, H^0\left(q^{-1}(U),\, (TE_G)\big\vert_{q^{-1}(U)}\right)
$$
be the horizontal lift of $v$ for the holomorphic connection $\theta'$ on $E_G$
in \eqref{e10}. Let
$$
[\widetilde{v},\, w]'\, \in\,
H^0\left(q^{-1}(U),\, {\rm kernel}(dq)\big\vert_{q^{-1}(U)}\right)\,
\subset\, H^0\left(q^{-1}(U),\, (TE_G)\big\vert_{q^{-1}(U)}\right)
$$
be the projection of the Lie bracket $[\widetilde{v},\, w]$ to the vertical component
for the connection $\theta'$ in \eqref{e10}. If $w$ is horizontal for the connection
$\theta'$, then we have $[\widetilde{v},\, w]'\,=\, 0$, because the connection $\theta'$
is flat which means that the horizontal distribution for $\theta'$ is integrable. Therefore,
for any $$w_1,\, w_2\, \in\, H^0(q^{-1}(U),\, TE_G\big\vert_{q^{-1}(U)})^G,$$ 
we have
$$
[\widetilde{v},\, w_1]'\, =\,[\widetilde{v},\, w_2]'
$$
if the vertical components of $w_1$ and $w_2$ coincide.
We noted above that the homomorphism in \eqref{b3} coincides with the homomorphism $\vartheta$ in \eqref{e-3}.
Now the proposition follows by comparing the constructions of $\Theta'$ and $\Theta$.
\end{proof}

\section{Deformations of Cartan geometry}\label{se4}

Let $S$ be a complex space. A \textit{holomorphic family} of Cartan geometries of type
$G/H$ parametrized by $S$ consists of the following:
\begin{enumerate}
\item $\phi\, :\, M_S\, \longrightarrow\, S$ is a holomorphic family of compact complex
manifolds parametrized by $S$.

\item $F\, :\, {\mathcal E}_H\, \longrightarrow\, M_S$ is a holomorphic principal 
$H$--bundle. The relative holomorphic tangent bundle for the projection $\phi\circ F$ will 
be denoted by $\mathcal T$. So $\mathcal T$ is the subbundle of $T{\mathcal E}_H$ given by 
the kernel of the differential $d(\phi\circ F)$ of the map $\phi\circ F$.

\item $\theta_S\, :\, {\mathcal T}\, \longrightarrow\, {\mathcal E}_H\times {\mathfrak g}$ is a
holomorphic isomorphism of ${\mathcal T}$ with the trivial holomorphic bundle
${\mathcal E}_H\times {\mathfrak g}\, \longrightarrow\, {\mathcal E}_H$ satisfying the
following two conditions:
\begin{itemize}
\item $\theta_S$ is $H$--equivariant for the action of $H$ on $T{\mathcal E}_H$ given by the
action of $H$ on ${\mathcal E}_H$ and the diagonal action of $H$ on ${\mathcal E}_H\times
{\mathfrak g}$ constructed using the adjoint action of $H$ on ${\mathfrak g}$ and the
action of $H$ on ${\mathcal E}_H$, and

\item the restriction of $\theta_S$ to $\text{kernel}(dF)$, where $dF$ is the differential
of the projection $F$, coincides with the
identification of $\text{kernel}(dF)$ with ${\mathcal E}_H\times \mathfrak h$ given by the
action of $H$ on ${\mathcal E}_H$.
\end{itemize}
\end{enumerate}

It should be clarified that the above definition is more general than
\cite[p.~516, Definition 2.2]{BDS}. In \cite[Definition 2.2]{BDS} the family
of complex manifolds $\phi\, :\, M_S\, \longrightarrow\, S$ is taken to be a
constant family of the form $M\times S\, \longrightarrow\, S$.

Let $(E_H,\, \theta)$ be a holomorphic Cartan geometry of type $G/H$ on a compact
complex manifold $M$. Let $S$ be a complex space with a distinguished point
$s_0\, \in\, S$. A \textit{deformation} of $(M,\, E_H,\, \theta)$ parametrized
by $S$ is a holomorphic family of Cartan geometries
$(M_S,\, \phi,\, {\mathcal E}_H,\, F,\, \theta_S)$ of type $G/H$ parametrized
by $S$ (see above) together with
\begin{itemize}
\item a holomorphic isomorphism of $M$ with $M_{s_0}\, :=\, \phi^{-1}(s_0)$, and

\item a holomorphic isomorphism of principal $H$--bundles $$E_H\,
\stackrel{\sim}{\longrightarrow}\, {\mathcal E}_H\big\vert_{M_{s_0}}$$ that takes
$\theta$ to the restriction of $\theta_S$ to $(\phi\circ F)^{-1}(s_0)$.
\end{itemize}

An \textit{isomorphism} between two deformations $(M_S,\, \phi,\, {\mathcal E}_H,\, F,\, 
\theta_S)$ and $(M'_S,\, \phi',\, {\mathcal E}'_H,\, F',\, \theta'_S)$ of
$(M,\, E_H,\, \theta)$ consists of a holomorphic isomorphism 
$M_S\, \stackrel{\sim}{\longrightarrow}\, M'_S$ parametrized by $S$ together with
a holomorphic isomorphism of principal $H$--bundles
$$
\delta\, :\, {\mathcal E}_H\, \stackrel{\sim}{\longrightarrow}\,{\mathcal E}'_H
$$
satisfying the following conditions:
\begin{itemize}
\item the isomorphism ${\mathcal T}\, \stackrel{\sim}{\longrightarrow}\, {\mathcal T}'$ 
given by $\delta$, where ${\mathcal T}'$ relative holomorphic tangent bundle
for the projection $\phi'\circ F'$, takes $\theta_S$ to $\theta'_S$,

\item the diagram
$$
\begin{matrix}
M& \stackrel{\rm Id}{\longrightarrow} & M\\
\Big\downarrow &&\Big\downarrow \\
M_S& \longrightarrow & M'_S
\end{matrix}
$$
commutes, and

\item the diagram
$$
\begin{matrix}
E_H& \stackrel{\rm Id}{\longrightarrow} & E_H\\
\Big\downarrow && \Big\downarrow \\
{\mathcal E}_H& \stackrel{\delta}{\longrightarrow} &{\mathcal E}'_H
\end{matrix}
$$
commutes.
\end{itemize}

When the parameter space $S$ is the nonreduced space ${\rm Spec}\, \mathbb{C}[t]/t^2$, then the deformation of
$(M,\, E_H,\, \theta)$ is called an \textit{infinitesimal deformation}.

Let $(E_H,\, \theta)$ be a holomorphic Cartan geometry of type $G/H$ on a compact
complex manifold $M$. Consider the homomorphism $\Theta$ in \eqref{e8}. Let
${\mathcal C}_\bullet$ denote the two-term complex of sheaves on $M$
$$
{\mathcal C}_\bullet\, \, :\,\, {\mathcal C}_0\,:=\, \text{At}(E_H)\,
\stackrel{\Theta}{\longrightarrow}\, {\mathcal C}_1\,:=\,\text{ad}(E_G)\otimes\Omega^1_M\, ,
$$
where ${\mathcal C}_i$ is at the $i$-th position. Using the inclusion map $h_1$
in \eqref{e-1} we get the following short exact sequence of complexes of sheaves on $M$
\begin{equation}\label{ec}
\begin{matrix}
&& 0 && 0\\
&& \Big\downarrow && \Big\downarrow\\
{\mathcal C}'_\bullet & : & {\mathcal C}'_0\,:=\, \text{ad}(E_H) &
\stackrel{\Theta}{\longrightarrow} & {\mathcal C}'_1\,:=\,\text{ad}(E_G)\otimes\Omega^1_M\\
&& \,\,\, \Big\downarrow h_1 && \,\,\,\,\,\, \Big\downarrow {\rm Id}\\
{\mathcal C}_\bullet & : & {\mathcal C}_0\,=\, \text{At}(E_H) &
\stackrel{\Theta}{\longrightarrow} & {\mathcal C}_1\,=\,\text{ad}(E_G)\otimes\Omega^1_M\\
&& \,\,\,\, \Big\downarrow df && \Big\downarrow\\
&& TM & \longrightarrow & 0\\
&& \Big\downarrow && \Big\downarrow\\
&& 0 && 0
\end{matrix}
\end{equation}
(the restriction of $\Theta$ to $\text{ad}(E_H)\, \subset\, \text{At}(E_H)$ is also denoted
by $\Theta$). From the short exact sequence of complexes in \eqref{ec} we get the following long
exact sequence of hypercohomologies
\begin{equation}\label{ec2}
\longrightarrow\, {\mathbb H}^1({\mathcal C}'_\bullet)\,\stackrel{\gamma_1}{\longrightarrow}\,
{\mathbb H}^1({\mathcal C}_\bullet)\,\stackrel{\gamma_2}{\longrightarrow}\,
H^1(M,\, TM)\, \longrightarrow \cdots .
\end{equation}

\begin{proposition}\label{prop2}\mbox{}
\begin{enumerate}
\item The space of all infinitesimal deformations of the holomorphic Cartan geometry
$(M,\, E_H,\, \theta)$ are parametrized by the first hypercohomology
${\mathbb H}^1({\mathcal C}_\bullet)$.

\item The space of infinitesimal deformations of the holomorphic Cartan geometry
$(E_H,\, \theta)$, keeping the complex manifold $M$ fixed, are parametrized by the
first hypercohomology ${\mathbb H}^1({\mathcal C}'_\bullet)$.

\item The homomorphism $\gamma_2$ in \eqref{ec2} is the natural map that sends
an infinitesimal deformation of $(M,\, E_H,\, \theta)$ to the infinitesimal deformation of
$M$ obtained from it by simply forgetting $E_H$ and $\theta$. In other words, $\gamma_2$ gives
the infinitesimal deformation of the underlying compact complex manifold
when the Cartan geometry $(M,\, E_H,\, \theta)$ deforms.

\item The homomorphism $\gamma_1$ in \eqref{ec2} is the natural map that sends
an infinitesimal deformation of $(E_H,\, \theta)$ to the
infinitesimal deformation of $(M,\, E_H,\, \theta)$ obtained from it by keeping
the complex manifold $M$ unchanged.
\end{enumerate}
\end{proposition}

\begin{proof}
This is a straightforward consequence of the computations in \cite{BDS} and \cite{Ch}.
\end{proof}

Consider the isomorphism $\vartheta\, :\, {\rm At}(E_H)\, \longrightarrow\, {\rm ad}(E_G)$
given by $(E_H,\, \theta)$ (see \eqref{e-3}). Using Proposition \ref{prop1}
and \eqref{e12}, the diagram in \eqref{ec} is transformed to the diagram
\begin{equation}\label{e16}
\begin{matrix}
&& 0 && 0\\
&& \Big\downarrow && \Big\downarrow\\
{\mathcal B}'_\bullet & : & {\mathcal B}'_0\,:=\, \text{ad}(E_H) &
\stackrel{\Theta'}{\longrightarrow} & {\mathcal B}'_1\,:=\,\text{ad}(E_G)\otimes\Omega^1_M\\
&& \Big\downarrow && \Big\downarrow\\
{\mathcal B}_\bullet & : & {\mathcal B}_0\,:=\, \text{ad}(E_G) &
\stackrel{\Theta'}{\longrightarrow} & {\mathcal B}_1\,:=\,\text{ad}(E_G)\otimes\Omega^1_M\\
&& \Big\downarrow && \Big\downarrow\\
&& TM\,=\, \text{ad}(E_G)/\text{ad}(E_H) & \longrightarrow & 0\\
&& \Big\downarrow && \Big\downarrow\\
&& 0 && 0
\end{matrix} 
\end{equation}
where $\Theta'$ is the homomorphism in \eqref{e11}; the restriction of
$\Theta'$ to $\text{ad}(E_H)\, \subset\, \text{ad}(E_G)$ is also denoted by
$\Theta'$. Let
\begin{equation}\label{e14}
\longrightarrow\, {\mathbb H}^1({\mathcal B}'_\bullet)\,\stackrel{\alpha_1}{\longrightarrow}\,
{\mathbb H}^1({\mathcal B}_\bullet)\,\stackrel{\alpha_2}{\longrightarrow}\,
H^1(M,\, TM)\,=\, H^1(M,\, \text{ad}(E_G)/\text{ad}(E_H))\, \longrightarrow \cdots
\end{equation}
be the corresponding long exact sequence of hypercohomologies.

The following is a reformulation of Proposition \ref{prop2}.

\begin{corollary}\label{cor1}\mbox{}
\begin{enumerate}
\item The space of all infinitesimal deformations of the holomorphic Cartan geometry
$(M,\, E_H,\, \theta)$ is parametrized by the first hypercohomology
${\mathbb H}^1({\mathcal B}_\bullet)$.

\item The space of infinitesimal deformations of the holomorphic Cartan geometry
$(E_H,\, \theta)$, keeping the complex manifold $M$ fixed, is parametrized by the first
hypercohomology ${\mathbb H}^1({\mathcal B}'_\bullet)$.

\item The homomorphism $\alpha_2$ in \eqref{e14} is the natural map that sends
an infinitesimal deformation of $(M,\, E_H,\, \theta)$ to the infinitesimal deformation of
$M$ obtained from it by simply forgetting $E_H$ and $\theta$.

\item The homomorphism $\alpha_1$ in \eqref{e14} is the natural map that sends
an infinitesimal deformation of $(E_H,\, \theta)$ to the
infinitesimal deformation of $(M,\, E_H,\, \theta)$ obtained from it by keeping
the complex manifold $M$ unchanged.
\end{enumerate}
\end{corollary}

\section{Flat Cartan geometry}

Let $(M_S,\, \phi,\, {\mathcal E}_H,\, F,\, \theta_S)$ be a holomorphic family of Cartan 
geometries of type $G/H$ parametrized by $S$; see Section \ref{se4}. Let ${\mathcal E}_G\, 
\longrightarrow\, M_S$ be the holomorphic principal $G$--bundle obtained by extending the 
structure group of ${\mathcal E}_H$ using the inclusion map of $H$ in $G$. Recall the 
construction of the holomorphic connection $\theta'$ in \eqref{e10}, given any holomorphic 
Cartan geometry of type $G/H$. This construction produces a relative holomorphic 
connection on the holomorphic principal $G$--bundle ${\mathcal E}_G$, relative for the 
projection ${\mathcal E}_G\,\longrightarrow\, S$. This relative
holomorphic connection on ${\mathcal E}_G$ will be denoted by $\theta'_S$.

Let $(E_H,\, \theta)$ be a flat holomorphic Cartan geometry of type $G/H$ on a compact 
complex manifold $M$. Let $S$ be a complex space. A \textit{holomorphic family of flat 
Cartan geometries} of type $G/H$ parametrized by $S$ is a holomorphic family of Cartan geometries 
$(M_S,\, \phi,\, {\mathcal E}_H,\, F,\, \theta_S)$ of type $G/H$ parametrized by $S$ such that
the relative holomorphic connection $\theta'_S$ on ${\mathcal E}_G$ is flat.

Associating to any flat holomorphic Cartan geometry $(E_H,\, \theta)$ of type $G/H$ on $M$,
we have a flat principal $G$--bundle $(E_G,\, \theta')$ on the topological
manifold $M$, where $\theta'$
is constructed in \eqref{e10}. Let
\begin{equation}\label{icg}
{\mathcal I}_{CG}
\end{equation}
denote the space of all
infinitesimal deformations of the flat holomorphic Cartan geometry $(M,\, E_H,\, \theta)$
in the category of the flat holomorphic Cartan geometries.
Let
\begin{equation}\label{ifc}
{\mathcal I}_{FC}
\end{equation}
denote the space of all
infinitesimal deformations of the flat principal $G$--bundle $(E_G,\, \theta')$ on
the topological manifold $M$. The above association of a flat $G$--bundle $(E_G,\, \theta')$
to a flat holomorphic Cartan geometry $(E_H,\, \theta)$ produces a homomorphism
\begin{equation}\label{e13}
\varphi\, :\, {\mathcal I}_{CG}\, \longrightarrow\, {\mathcal I}_{FC}
\end{equation}
(see \eqref{icg} and \eqref{ifc}).

\begin{theorem}\label{thm1}
The homomorphism $\varphi$ in \eqref{e13} is an isomorphism.
\end{theorem}

\begin{proof}
As before, $(E_H,\, \theta)$ is a flat holomorphic Cartan geometry of type $G/H$ on a
compact complex manifold $M$. Consider the holomorphic connection $\Theta'$ on
$\text{ad}(E_G)$ in \eqref{e11} given by $(E_H,\, \theta)$. We note that $\Theta'$ is
flat because the holomorphic connection $\theta'$ on $E_G$ in \eqref{e10} is flat.
Since $\Theta'$ is flat, it produces the following complex
$\widetilde{\mathcal B}_\bullet$ of sheaves on $M$:
\begin{equation}\label{e15}
\widetilde{\mathcal B}_\bullet\, \, :\,\, \widetilde{\mathcal B}_0\,:=\, \text{ad}(E_G)\, 
\stackrel{\Theta'}{\longrightarrow}\, \widetilde{\mathcal B}_1
\,:=\,\text{ad}(E_G)\otimes\Omega^1_M \,
\stackrel{\Theta'}{\longrightarrow}\, \widetilde{\mathcal B}_2
\,:=\,\text{ad}(E_G)\otimes\Omega^2_M.
\end{equation}
We note that this $\widetilde{\mathcal B}_\bullet$ and the complex
${\mathcal B}_\bullet$ in \eqref{e16} fit in the following short exact
sequences of complexes of sheaves on $M$:
$$
\begin{matrix}
&& 0 && 0 && 0\\
&& \Big\downarrow && \Big\downarrow&& \Big\downarrow\\
{\mathcal D}_\bullet & : & {\mathcal D}_0\, :=\,0 &\longrightarrow &
{\mathcal D}_1\, :=\,0 & \longrightarrow & {\mathcal D}_2\, :=\, \text{ad}(E_G)\otimes\Omega^2_M\\
&& \Big\downarrow && \Big\downarrow&& \Big\downarrow\\
\widetilde{\mathcal B}_\bullet & : & \widetilde{\mathcal B}_0\,:=\, \text{ad}(E_G) &
\stackrel{\Theta'}{\longrightarrow} & \widetilde{\mathcal B}_1\,:=\,\text{ad}(E_G)
\otimes\Omega^1_M & \stackrel{\Theta'}{\longrightarrow} & \widetilde{\mathcal B}_2\,:=\,
\text{ad}(E_G)\otimes\Omega^2_M\\
&& \Big\downarrow && \Big\downarrow&& \Big\downarrow\\
{\mathcal B}_\bullet & : & {\mathcal B}_0\,:=\, \text{ad}(E_G) &
\stackrel{\Theta'}{\longrightarrow} & {\mathcal B}_1\,:=\,\text{ad}(E_G)\otimes\Omega^1_M
& \longrightarrow & 0 \\
&& \Big\downarrow && \Big\downarrow&& \Big\downarrow\\
&& 0 && 0&& 0
\end{matrix}
$$
Let
\begin{equation}\label{e17}
{\mathbb H}^1({\mathcal D}_\bullet) \, \longrightarrow\, {\mathbb H}^1(\widetilde{\mathcal
B}_\bullet)\, \stackrel{\Phi}{\longrightarrow}\, {\mathbb H}^1({\mathcal B}_\bullet)
\end{equation}
be the corresponding long exact sequence of hypercohomologies. Since
$$
{\mathbb H}^i({\mathcal D}_\bullet)\,=\, H^{i-2}(M,\, \text{ad}(E_G)\otimes\Omega^2_M),
$$
we have ${\mathbb H}^1({\mathcal D}_\bullet)\,=\, 0$, and hence \eqref{e17} gives
an injective homomorphism
$$
\Phi\, :\, {\mathbb H}^1(\widetilde{\mathcal B}_\bullet) \, \longrightarrow\,
{\mathbb H}^1({\mathcal B}_\bullet)\, .
$$

{}From Corollary \ref{cor1}(1) we know that ${\mathbb H}^1({\mathcal B}_\bullet)$
coincides with the space of all infinitesimal deformations of the holomorphic Cartan geometry
$(M,\, E_H,\, \theta)$. Now, ${\mathbb H}^1(\widetilde{\mathcal B}_\bullet)$ coincides with
the space of all infinitesimal deformations of the flat holomorphic Cartan geometry
$(M,\, E_H,\, \theta)$ in the category of the flat holomorphic Cartan geometries. The
injective homomorphism $\Phi$ in \eqref{e17} is the natural
map that considers an infinitesimal deformation of the flat holomorphic Cartan geometry
$(M,\, E_H,\, \theta)$ in the category of the flat holomorphic
Cartan geometries as simply an infinitesimal deformation of the
holomorphic Cartan geometry $(M,\, E_H,\, \theta)$.

The kernel of the homomorphism
$$
\Theta'\, :\, {\rm ad}(E_G)\, \longrightarrow\,\text{ad}(E_G)\otimes\Omega^1_M
$$
is the local system on $M$ given by the sheaf of flat sections of ${\rm ad}(E_G)$
for the flat connection $\Theta'$. This locally constant sheaf of flat sections of
${\rm ad}(E_G)$ will be denoted by $\underline{\rm ad}(E_G)$. The
space of infinitesimal deformations ${\mathcal I}_{FC}$ of the flat connection (see
\eqref{e13}) has the following description
\begin{equation}\label{e18}
{\mathcal I}_{FC}\,=\, H^1(M,\, \underline{\rm ad}(E_G))
\end{equation}
(see \cite{Go}, \cite{GM}).

From the homomorphism of complexes of sheaves
$$
\begin{matrix}
&&\underline{\rm ad}(E_G) &\longrightarrow & 0 &\longrightarrow & 0\\
&&\Big\downarrow && \Big\downarrow && \Big\downarrow\\
\widetilde{\mathcal B}_{\bullet} & : & \widetilde{\mathcal B}_0\,:=\, \text{ad}(E_G) &
\stackrel{\Theta'}{\longrightarrow} & \widetilde{\mathcal B}_1\,:=\,
\text{ad}(E_G)\otimes\Omega^1_M
& \stackrel{\Theta'}{\longrightarrow} & \widetilde{\mathcal B}_2\,:=\,
\text{ad}(E_G)\otimes\Omega^2_M\\
\end{matrix}
$$
we conclude that
$$
H^1(M,\, \underline{\rm ad}(E_G))\,=\,{\mathbb H}^1(\widetilde{\mathcal B}_\bullet)\, .
$$
Now the theorem follow from \ref{e18} and Corollary \ref{cor1}(1).
\end{proof}

Using the Riemann--Hilbert correspondence, the space ${\mathcal I}_{FC}$ in \eqref{ifc}
is identified with the the space of all
infinitesimal deformations of the pair $(E_G,\, \theta')$ in the category of flat
holomorphic $G$--connections on the fixed complex manifold $M$.

\section*{Acknowledgements}

We are very grateful to the referee for helpful comments.
We thank Yasuhiro Wakabayashi who pointed out that the kernel of the exact sequence in the statement of
the Theorem 3.4 in \cite{BDS} is not the correct one (see \cite{Wa} Remark 6.4.5).

The first two authors were partially supported by the French government through the UCAJEDI Investments in the 
Future project managed by the National Research Agency (ANR) with the reference number 
ANR2152IDEX201. The first-named author is partially supported by a J. C. Bose Fellowship, and 
School of Mathematics, TIFR, is supported by 12-R$\&$D-TFR-5.01-0500. 


\end{document}